\documentclass[reqno,12pt]{amsart}
\usepackage{amsfonts}
\usepackage{bbm}
\usepackage{}
\numberwithin{equation}{section}
\usepackage{color}
\usepackage{graphicx}
\usepackage{indentfirst}
\usepackage{color}
\usepackage{amssymb}
\usepackage{mathrsfs}
\usepackage{xy}
\xyoption{all}
 \def\Hom{\mbox{\rm Hom}}  
 \def\fin{\hfill$\square$}   \def\mod{\mbox{\rm \textbf{mod}}\,}
    
\def\cone{\mbox{\rm cone}}\def\cocone{\mbox{\rm cocone}}
\def\cim{\mbox{\rm sim}}\def\add{\mbox{\rm add}}
 
\def\Filt{\mbox{\rm \textbf{Filt}}}

\theoremstyle{plain}
\newtheorem{theorem}{\bf Theorem}[section]
\newtheorem{lemma}[theorem]{\bf Lemma}
\newtheorem{corollary}[theorem]{\bf Corollary}
\newtheorem{proposition}[theorem]{\bf Proposition}

\theoremstyle{definition}
\newtheorem{definition}[theorem]{\bf Definition}
\newtheorem{remark}[theorem]{\bf Remark}
\newtheorem{example}[theorem]{\bf Example}

\newcommand{\bt}{\begin{theorem}}
\newcommand{\et}{\end{theorem}}
\newcommand{\bl}{\begin{lemma}}
\newcommand{\el}{\end{lemma}}
\newcommand{\bd}{\begin{definition}}
\newcommand{\ed}{\end{definition}}
\newcommand{\bc}{\begin{corollary}}
\newcommand{\ec}{\end{corollary}}
\newcommand{\bp}{\begin{proof}}
\newcommand{\ep}{\end{proof}}
\newcommand{\bx}{\begin{example}}
\newcommand{\ex}{\end{example}}
\newcommand{\br}{\begin{remark}}
\newcommand{\er}{\end{remark}}
\newcommand{\be}{\begin{equation}}
\newcommand{\ee}{\end{equation}}
\newcommand{\ba}{\begin{align}}
\newcommand{\ea}{\end{align}}
\newcommand{\bn}{\begin{enumerate}}
\newcommand{\en}{\end{enumerate}}
\newcommand{\bcs}{\begin{cases}}
\newcommand{\ecs}{\end{cases}}

\makeatletter
\renewcommand{\section}{\@startsection{section}{1}{0mm}
  {-\baselineskip}{0.5\baselineskip}{\bf\leftline}}
\makeatother

\begin{document}

\title[ Semibricks in extriangulated categories]{Semibricks in extriangulated categories}
\author[L. Wang, J. Wei, H. Zhang]{Li Wang, Jiaqun Wei,  Haicheng Zhang}
\address{Institute of Mathematics, School of Mathematical Sciences, Nanjing Normal University,
 Nanjing 210023, P. R. China.\endgraf}
\email{wl04221995@163.com (Wang); weijiaqun@njnu.edu.cn (Wei); zhanghc@njnu.edu.cn (Zhang).}


\begin{abstract}
Let $\mathcal{X}$ be a semibrick in an extriangulated category $\mathscr{C}$. Let $\mathcal{T}$ be the filtration subcategory generated by $\mathcal{X}$. We give a one-to-one correspondence between simple semibricks and length wide subcategories in $\mathscr{C}$. This generalizes a bijection given by Ringel in module categories, which has been generalized by Enomoto to exact categories. Moreover, we also give a one-to-one correspondence between cotorsion pairs in $\mathcal{T}$ and certain subsets of $\mathcal{X}$. Applying to the simple minded systems of an triangulated category, we recover a result given by Dugas.

\end{abstract}

\maketitle

\section{Introduction}

In representation theory of a finite-dimensional algebra $A$ over a field, the notion of
simple modules is fundamental. By Schur's lemma, the endomorphism ring of a simple module is a division algebra; and there exists no nonzero homomorphism between two nonisomorphic simple modules. For each $A$-module satisfying that its endomorphism ring is a division algebra, we call it a brick. This
notion is a generalization of simple modules. For each set of
isoclasses of pairwise Hom-orthogonal bricks, we call it a semibrick. By Ringel \cite{Ri}, semibricks of $A$-modules
correspond bijectively to the wide subcategories of $\mod A$, that is, the subcategories which
are closed under taking kernels, cokernels, and extensions. It is noted that bricks and wide subcategories have close relationship with ring epimorphisms and
universal localizations (cf. \cite{St,Sch,GL}). Asai studies the semibricks from the point of view of $\tau$-tilting theory in \cite{As}.

Recently, Enomoto \cite{En} generalizes the notion of simple objects in an abelian category to an exact category and then generalizes Ringel's bijection to exact categories. The notion of a triangulated category was introduced by Grothendieck and later by Verdier \cite{Ve}. It has become a very powerful tool in many branches of mathematics, and has been investigated by many papers such as \cite{Th}, \cite{Ha}, \cite{Ne}, \cite{Iy2}. Recently, Nakaoka
and Palu \cite{Na} introduced an extriangulated category which is extracting properties on triangulated categories and exact categories.

In this paper, we study the semibricks in an extriangulated category. Explicitly, we introduce simple objects, wide subcategories in an extriangulated category, and then
generalizes Ringel's bijection to extriangulated categories. Moreover, we also establish a relation between semibricks and cotorsion pairs.
%
%
%
%
%

The paper is organized as follows: We summarize some basic definitions and properties of an extriangulated category and its filtration subcategory in Section 2. In Section 3, we introduce wide subcategories of an extriangulated category, and give some properties of semibricks. Section 4 is devoted to giving a one-to-one correspondence between simple semibricks and length wide subcategories.  Finally, we study a relation between filtration subcategories generated by a semibrick and cotorsion pairs in Section 5.

\subsection{Conventions and notation.}
A subcategory $\mathcal{D}$ of an additive category $\mathscr{C}$ is said to be {\em contravariantly finite} in $\mathscr{C}$ if for each object $M\in\mathscr{C}$, there exists a morphism $f:X\rightarrow M$ with $X\in \mathcal{D}$ such that $\mathscr{C}(\mathcal{D},f)$ is an epimorphism. Dually, we define {\em covariantly
finite} subcategories in $\mathscr{C}$. Furthermore, a subcategory of $\mathscr{C}$ is said to be {\em functorially finite} in $\mathscr{C}$ if it is both contravariantly finite and covariantly finite in $\mathscr{C}$. An additive category is {\em Krull--Schmidt} if each of its objects is the direct sum of finitely many objects with local endomorphism rings.

Throughout this paper, we assume, unless otherwise stated, that all considered categories are skeletally small, Hom-finite, Krull--Schmidt, $k$-linear over a fixed field $k$, and subcategories are {full} and closed under isomorphisms. We write $\mathcal{X}\subseteq\mathscr{C}$ for a subset of objects in $\mathscr{C}$, which we identify with the
corresponding full subcategory of $\mathscr{C}$. Let $Q$ a finite acyclic quiver, we denote by $S_i$ the one-dimensional simple (left) $kQ$-module associated to the vertex $i$ of $Q$. Denote by $P_i$ and $I_i$ the projective cover and injective envelop of $S_i$, respectively.

\section{Preliminaries}

\subsection{Extriangulated categories}
Let us recall some notions concerning extriangulated categories from \cite{Na}.

Let $\mathscr{C}$ be an additive category and let $\mathbb{E}$: $\mathscr{C}^{op}\times\mathscr{C}\rightarrow Ab$ be a biadditive functor. For any pair of objects $A$, $C\in\mathscr{C}$, an element $\delta\in \mathbb{E}(C,A)$ is called an {\em $\mathbb{E}$-extension}. The zero element $0\in\mathbb{E}(C,A)$ is called the {\em split $\mathbb{E}$-extension}.
 For any morphism $a\in \mathscr{C}(A,A')$ and $c\in {\mathscr{C}}(C',C)$, we have
$\mathbb{E}(C,a)(\delta)\in\mathbb{E}(C,A')$ and $\mathbb{E}(c,A)(\delta)\in\mathbb{E}(C',A).$
We simply denote them by $a_{\ast}\delta$ and $c^{\ast}\delta$,~respectively.  A morphism $(a,c)$: $\delta\rightarrow\delta'$ of $\mathbb{E}$-extensions is a pair of morphisms $a\in \mathscr{C}(A,A')$ and $c\in {\mathscr{C}}(C,C')$ satisfying the equality $a_{\ast}\delta=c^{\ast}\delta'$.

By Yoneda's lemma, any $\mathbb{E}$-extension $\delta\in \mathbb{E}(C,A)$ induces natural transformations
$$\delta_{\sharp}: \mathscr{C}(-,C)\rightarrow\mathbb{E}(-,A)~~\text{and}~~\delta^{\sharp}: \mathscr{C}(A,-)\rightarrow\mathbb{E}(C,-).$$ For any $X\in\mathscr{C}$, these
$(\delta_{\sharp})_X$ and $(\delta^{\sharp})_X$ are defined by
$(\delta_{\sharp})_X:  \mathscr{C}(X,C)\rightarrow\mathbb{E}(X,A), f\mapsto f^\ast\delta$ and $(\delta^{\sharp})_X:   \mathscr{C}(A,X)\rightarrow\mathbb{E}(C,X), g\mapsto g_\ast\delta$.

Two sequences of morphisms $A\stackrel{x}{\longrightarrow}B\stackrel{y}{\longrightarrow}C$ and $A\stackrel{x'}{\longrightarrow}B'\stackrel{y'}{\longrightarrow}C$ in $\mathscr{C}$ are said to be {\em equivalent} if there exists an isomorphism $b\in \mathscr{C}(B,B')$ such that the following diagram
$$\xymatrix{
  A \ar@{=}[d] \ar[r]^-{x} & B\ar[d]_{b}^-{\simeq} \ar[r]^-{y} & C\ar@{=}[d] \\
  A \ar[r]^-{x'} &B' \ar[r]^-{y'} &C  }$$ is commutative.
We denote the equivalence class of $A\stackrel{x}{\longrightarrow}B\stackrel{y}{\longrightarrow}C$ by $[A\stackrel{x}{\longrightarrow}B\stackrel{y}{\longrightarrow}C]$. In addition, for any $A,C\in\mathscr{C}$, we denote as
$$0=[A\stackrel{\tiny\begin{pmatrix} 1\\0\end{pmatrix}}{\longrightarrow}A\oplus C\stackrel{\tiny\begin{pmatrix} 0&1\end{pmatrix}}{\longrightarrow}C].$$
For any two classes $[A\stackrel{x}{\longrightarrow}B\stackrel{y}{\longrightarrow}C]$ and $[A'\stackrel{x'}{\longrightarrow}B'\stackrel{y'}{\longrightarrow}C']$, we denote as
$$[A\stackrel{x}{\longrightarrow}B\stackrel{y}{\longrightarrow}C]\oplus[A'\stackrel{x'}{\longrightarrow}B'\stackrel{y'}{\longrightarrow}C']=
[A\oplus A'\stackrel{x\oplus x'}{\longrightarrow}B\oplus B'\stackrel{y\oplus y'}{\longrightarrow}C\oplus C'].$$

\begin{definition}
Let $\mathfrak{s}$ be a correspondence which associates an equivalence class $\mathfrak{s}(\delta)=[A\stackrel{x}{\longrightarrow}B\stackrel{y}{\longrightarrow}C]$ to any $\mathbb{E}$-extension $\delta\in\mathbb{E}(C,A)$ . This $\mathfrak{s}$ is called a {\em realization} of $\mathbb{E}$ if for any morphism $(a,c):\delta\rightarrow\delta'$ with $\mathfrak{s}(\delta)=[\Delta_{1}]$ and $\mathfrak{s}(\delta')=[\Delta_{2}]$, there is a commutative diagram as follows:
$$\xymatrix{
\Delta_{1}\ar[d] & A \ar[d]_-{a} \ar[r]^-{x} & B  \ar[r]^{y}\ar[d]_-{b} & C \ar[d]_-{c}    \\
 \Delta_{2}&A\ar[r]^-{x'} & B \ar[r]^-{y'} & C .   }
$$  A realization $\mathfrak{s}$ of $\mathbb{E}$ is said to be {\em additive} if it satisfies the following conditions:

(a) For any $A,~C\in\mathscr{C}$, the split $\mathbb{E}$-extension $0\in\mathbb{E}(C,A)$ satisfies $\mathfrak{s}(0)=0$.

(b) $\mathfrak{s}(\delta\oplus\delta')=\mathfrak{s}(\delta)\oplus\mathfrak{s}(\delta')$ for any pair of $\mathbb{E}$-extensions $\delta$ and $\delta'$.
\end{definition}

Let $\mathfrak{s}$ be an additive realization of $\mathbb{E}$. If $\mathfrak{s}(\delta)=[A\stackrel{x}{\longrightarrow}B\stackrel{y}{\longrightarrow}C]$, then the sequence $A\stackrel{x}{\longrightarrow}B\stackrel{y}{\longrightarrow}C$ is called a {\em conflation}, $x$ is called an {\em inflation} and $y$ is called a {\em deflation}.
In this case, we say $A\stackrel{x}{\longrightarrow}B\stackrel{y}{\longrightarrow}C\stackrel{\delta}\dashrightarrow$ is an $\mathbb{E}$-triangle.
We will write $A=\cocone(y)$ and $C=\cone(x)$ if necessary. We say an $\mathbb{E}$-triangle is {\em splitting} if it
realizes 0.

\begin{definition}
(\cite[Definition 2.12]{Na})\label{F}
We call the triplet $(\mathscr{C}, \mathbb{E},\mathfrak{s})$ an {\em extriangulated category} if it satisfies the following conditions:\\
$\rm(ET1)$ $\mathbb{E}$: $\mathscr{C}^{op}\times\mathscr{C}\rightarrow Ab$ is a biadditive functor.\\
$\rm(ET2)$ $\mathfrak{s}$ is an additive realization of $\mathbb{E}$.\\
$\rm(ET3)$ Let $\delta\in\mathbb{E}(C,A)$ and $\delta'\in\mathbb{E}(C',A')$ be any pair of $\mathbb{E}$-extensions, realized as
$\mathfrak{s}(\delta)=[A\stackrel{x}{\longrightarrow}B\stackrel{y}{\longrightarrow}C]$, $\mathfrak{s}(\delta')=[A'\stackrel{x'}{\longrightarrow}B'\stackrel{y'}{\longrightarrow}C']$. For any commutative square in $\mathscr{C}$
$$\xymatrix{
  A \ar[d]_{a} \ar[r]^{x} & B \ar[d]_{b} \ar[r]^{y} & C \\
  A'\ar[r]^{x'} &B'\ar[r]^{y'} & C'}$$
there exists a morphism $(a,c)$: $\delta\rightarrow\delta'$ which is realized by $(a,b,c)$.\\
$\rm(ET3)^{op}$~Dual of $\rm(ET3)$.\\
$\rm(ET4)$~Let $\delta\in\mathbb{E}(D,A)$ and $\delta'\in\mathbb{E}(F,B)$ be $\mathbb{E}$-extensions realized by
$A\stackrel{f}{\longrightarrow}B\stackrel{f'}{\longrightarrow}D$ and $B\stackrel{g}{\longrightarrow}C\stackrel{g'}{\longrightarrow}F$, respectively.
Then there exist an object $E\in\mathscr{C}$, a commutative diagram
\begin{equation}\label{2.1}
\xymatrix{
  A \ar@{=}[d]\ar[r]^-{f} &B\ar[d]_-{g} \ar[r]^-{f'} & D\ar[d]^-{d} \\
  A \ar[r]^-{h} & C\ar[d]_-{g'} \ar[r]^-{h'} & E\ar[d]^-{e} \\
   & F\ar@{=}[r] & F   }
\end{equation}
in $\mathscr{C}$, and an $\mathbb{E}$-extension $\delta''\in \mathbb{E}(E,A)$ realized by $A\stackrel{h}{\longrightarrow}C\stackrel{h'}{\longrightarrow}E$, which satisfy the following compatibilities:\\
$(\textrm{i})$ $D\stackrel{d}{\longrightarrow}E\stackrel{e}{\longrightarrow}F$ realizes $\mathbb{E}(F,f')(\delta')$,\\
$(\textrm{ii})$ $\mathbb{E}(d,A)(\delta'')=\delta$,\\
$(\textrm{iii})$ $\mathbb{E}(E,f)(\delta'')=\mathbb{E}(e,B)(\delta')$.\\
$\rm(ET4)^{op}$ Dual of $\rm(ET4)$.
\end{definition}

Let $\mathscr{C}$ be an extriangulated category, and $\mathcal{D},\mathcal{D}'\subseteq\mathscr{C}$. We write $\mathcal{D}\ast\mathcal{D}'$ for the full subcategory of objects $X$ admitting an $\mathbb{E}$-triangle $D\stackrel{}{\longrightarrow}X\stackrel{}{\longrightarrow}D'\stackrel{}\dashrightarrow$ with $D\in\mathcal{D}$ and $D'\in\mathcal{D}'$. A subcategory $\mathcal{D}$ of  $\mathscr{C}$ is {\em extension-closed}, if $\mathcal{D}\ast\mathcal{D}=\mathcal{D}$.
An object $P$ in $\mathscr{C}$ is called {\em projective} if for any conflation $A\stackrel{x}{\longrightarrow}B\stackrel{y}{\longrightarrow}C$ and any morphism $c$ in $\mathscr{C}(P,C)$, there exists $b$ in $\mathscr{C}(P,B)$ such that $yb=c$. We denote the full subcategory of projective objects in $\mathscr{C}$ by $\mathcal{P}$. Dually, the {\em injective} objects are defined, and the full subcategory of injective objects in $\mathscr{C}$ is denoted by $\mathcal{I}$. We say that $\mathscr{C}$ {\em has enough projectives} if for any object $M\in\mathscr{C}$, there exists an $\mathbb{E}$-triangle $A\stackrel{}{\longrightarrow}P\stackrel{}{\longrightarrow}M\stackrel{}\dashrightarrow$ satisfying $P\in\mathcal{P}$. Dually, we define that $\mathscr{C}$ {\em has enough injectives}. In particular, if $\mathscr{C}$ is a triangulated category, then $\mathscr{C}$  has enough projectives and injectives with $\mathcal{P}$ and $\mathcal{I}$ consisting of zero objects.

\begin{example}
(a)  Exact categories, triangulated categories and extension-closed subcategories of triangulated categories are
extriangulated categories. (cf. \cite{Na})

(b) Let $\mathscr{C}$ be an extriangulated category. Then $\mathscr{C}/(\mathcal{P}\cap \mathcal{I})$ is an extriangulated category which is neither exact nor triangulated in general (cf. \cite[Proposition 3.30]{Na}).
\end{example}

\begin{proposition} \cite[Proposition 3.3]{Na}
Let $\mathscr{C}$ be an extriangulated category. For any $\mathbb{E}$-triangle $A\stackrel{}{\longrightarrow}B\stackrel{}{\longrightarrow}C\stackrel{\delta}\dashrightarrow$, the following sequences of natural transformations are exact.
$$\mathscr{C}(C,-)\rightarrow \mathscr{C}(B,-)\rightarrow \mathscr{C}(A,-)\stackrel{\delta^{\sharp}}\rightarrow\mathbb{E}(C,-)\rightarrow\mathbb{E}(B,-),$$
$$\mathscr{C}(-,A)\rightarrow \mathscr{C}(-,B)\rightarrow \mathscr{C}(-,C)\stackrel{\delta_{\sharp}}\rightarrow\mathbb{E}(-,A)\rightarrow\mathbb{E}(-,B).$$
\end{proposition}

\begin{lemma}\label{2.5} The upper-right square in (\ref{2.1}) is a weak pushout and weak pullback.
\end{lemma}
\begin{proof}By \cite[Lemma  3.13]{Na}, it follows that it is a weak pushout, so we only need to prove it is a weak pullback.  Let $x\in\mathscr{C}(M,C)$, $y\in\mathscr{C}(M,D)$ be two morphisms such that $h'x=dy$.

By $y^{\ast}\delta=y^{\ast}d^{\ast}\delta''=(dy)^{\ast}\delta''=(h'x)^{\ast}\delta''=x^{\ast}(h'^{\ast}\delta'')=0$ and the exactness of
$$\mathscr{C}(M,B)\stackrel{}{\longrightarrow}\mathscr{C}(M,D)\stackrel{}{\longrightarrow}\mathbb{E}(M,A),$$
there exists $l:M\rightarrow B$ such that $y=f'l$. Furthermore, by $h'(gl-x)=df'l-dy=0$ and the exactness of
 $$\mathscr{C}(M,A)\stackrel{}{\longrightarrow}\mathscr{C}(M,C)\stackrel{}{\longrightarrow}\mathscr{C}(M,E),$$
 there exists $s:M\rightarrow A$ such that $gl-x=hs$. Thus, we have obtained that $x=gl-hs=gl-gfs=g(l-fs)$ and $f'(l-fs)=y$. Hence, $t=l-fs$ makes the following diagram  communicative
 $$\xymatrix{
 M \ar@/_/[ddr]_{x} \ar@/^/[drr]^{y}
    \ar@{.>}[dr]|-{t}                   \\
   & B\ar[d]^{g} \ar[r]_{f'}
                      & D \ar[d]_{d}    \\
   & C \ar[r]^{h'}     & E.               }
$$\end{proof}

\subsection{Filtration subcategories}
In this subsection, let $\mathscr{C}$ be always an extriangulated category. For a collection $\mathcal{X}$ of objects in $\mathscr{C}$, we define full subcategories
$$^{\perp }\mathcal{X}=\{M\in \mathscr{C}|~\mathscr{C}( M,\mathcal{X})=0 \}$$ and
$$^{\perp_{1}}\mathcal{X}=\{M\in \mathscr{C}|~\mathbb{E}( M,\mathcal{X})=0 \}.$$
Dually, we define full subcategories $\mathcal{X}^{\perp }$ and $\mathcal{X}^{\perp_{1}}$ in $\mathscr{C}$.
The {\em filtration subcategory} $\mathbf{Filt_{\mathscr{C}}(\mathcal{X})}$ is consisting of all objects $M$ admitting a finite filtration of the from
\begin{equation}\label{2.2}
0=X_{0}\stackrel{f_{0}}{\longrightarrow}X_{1}\stackrel{f_{1}}{\longrightarrow}X_{2}{\longrightarrow}\cdots\stackrel{f_{n-1}}{\longrightarrow}X_{n}=M
\end{equation}
with $f_{i}$ being an inflation and $\cone(f_{i})\in\mathcal{X}$ for any $0\leq i\leq n-1$. In this case, we say that $M$ possesses an $\mathcal{X}$-{\em filtration} of length $n$ and the minimal length of such a filtration is called the $\mathcal{X}$-{\em length} of $M$, which is denoted by $l_{\mathcal{X}}(M)$.
\begin{remark}\label{2.6}
Let $\mathcal{X}$ be a collection of objects in $\mathscr{C}$. The filtration subcategory can be defined inductively as follows:

$(1)$ The filtration subcategory $\Filt_{\mathscr{C}}(\mathcal{X})=$$\cup_{n\in \mathbb{N}}F_{n}(X)$, where $F_{0}(X)=0$ and $F_{n}(X)=F_{n-1}(X)\ast (X\cup\{0\})$ for $n\geq1$. Observe that $F_{n-1}(X)\subseteq F_{n}(X)$ for $n\geq1$. Hence, $l_{\mathcal{X}}(M)=n$ if and only if $M\in F_{n}(X)$ but $M\notin F_{n-1}(X)$.

$(2)$ The filtration subcategory $\Filt_{\mathscr{C}}(\mathcal{X})=\cup_{n\in\mathbb{N}}F^{n}(X)$, where $F^{0}(X)=0$ and $F^{n}(X)=(X\cup\{0\})\ast F^{n-1}(X) $ for $n\geq1$. Noting that the operation $\ast$ is associative (cf. \cite[Lemma 3.9]{Zh}), we obtain that $F^{n}(X)= F_{n}(X)$. That is, an object $M$ admits an $\mathcal{X}$-filtration as (\ref{2.2}) if and only if there exists a finite filtration of the from
\begin{equation}M=Y_{n}\stackrel{g_{n-1}}{\longrightarrow}Y_{n-1}\stackrel{g_{n-2}}{\longrightarrow}Y_{n-2}{\longrightarrow}\cdots\stackrel{g_{0}}{\longrightarrow}Y_{0}=0\end{equation}
such that $g_{i}$ is a deflation and $\cocone(g_{i})\in\mathcal{X}$ for $0\leq i\leq n-1$.

$(3)$ Note that $^{\perp}\mathcal{X}=^{\perp}\mathbf{Filt_{\mathscr{C}}(\mathcal{X})}$ and $^{\perp_{1}}\mathcal{X}=^{\perp_{1}}\mathbf{Filt_{\mathscr{C}}(\mathcal{X})}$ (cf. \cite[Lemma 3.4]{Zh}). Dually, we have that $\mathcal{X}^{\perp}=\mathbf{Filt_{\mathscr{C}}(\mathcal{X})}^{\perp}$ and $\mathcal{X}^{\perp_{1}}=\mathbf{Filt_{\mathscr{C}}(\mathcal{X})}^{\perp_{1}}$.
\end{remark}

In what follows, we say that a commutative diagram is {\em exact} if every sub-diagram of the form $X\rightarrow Y\rightarrow Z$ is a conflation.

\begin{lemma}\label{1} Let $\mathcal{X}$ be a collection of objects in $\mathscr{C}$, and $A,C\in\mathbf{Filt_{\mathscr{C}}(\mathcal{X})}$. Then for any $\mathbb{E}$-triangle $A\stackrel{}{\longrightarrow}B\stackrel{}{\longrightarrow}C\stackrel{}\dashrightarrow$ in $\mathscr{C}$, we have that $B\in\mathbf{Filt_{\mathscr{C}}(\mathcal{X})}$ and $l_{\mathcal{X}}(B)\leq l_{\mathcal{X}}(A)+l_{\mathcal{X}}(C)$.
\end{lemma}
\begin{proof} Set $l_{\mathcal{X}}(A)=m$ and $l_{\mathcal{X}}(C)=n$. If $m=0$ or $n=0$, then the result is clear. So we assume that $m,n>0$.
Fix an $\mathcal{X}$-filtration of $A$
\begin{equation}\label{1.1}
0=X_{0}\stackrel{}{\longrightarrow}X_{1}\stackrel{}{\longrightarrow}X_{2}{\longrightarrow}\cdots{\longrightarrow}X_{m}=A.
\end{equation}
If $n=1$, i.e., $C\in\mathcal{X}$, then combining (\ref{1.1}) with the $\mathbb{E}$-triangle $A\stackrel{}{\longrightarrow}B\stackrel{}{\longrightarrow}C\stackrel{}\dashrightarrow$, we obtain that $l_{\mathcal{X}}(B)\leq m+1$. For $n\geq2$, take an $\mathcal{X}$-filtration of $C$
$$0=Y_{0}\stackrel{f_{0}}{\longrightarrow}Y_{1}\stackrel{f_{1}}{\longrightarrow}Y_{2}{\longrightarrow}\cdots\stackrel{f_{n-1}}{\longrightarrow}Y_{n}=C.$$
Now, we can form the following exact commutative diagram
$$\xymatrix{
  A \ar@{=}[d] \ar[r] & N_{1} \ar[d]^{g_{1}} \ar[r] & Y_{n-1} \ar[d]^{f_{n-1}}  \\
  A  \ar[r] & B \ar[d] \ar[r] & C \ar[d] \\
  & \cone(f_{n-1})  \ar@{=}[r] & \cone(f_{n-1}).   }
$$
That is, there exists $g_1: N_{1}\rightarrow B$ such that $\cone(g_{1})=\cone(f_{n-1})\in \mathcal{X}$.
Furthermore, we have the following exact commutative diagram
$$\xymatrix{
  A \ar@{=}[d] \ar[r] & N_{2} \ar[d]^{g_{2}} \ar[r] & Y_{n-2} \ar[d]^{f_{n-2}}  \\
  A  \ar[r] & N_{1} \ar[d] \ar[r] & Y_{n-1} \ar[d] \\
  & \cone(f_{n-2})  \ar@{=}[r] & \cone(f_{n-2})   }
$$
with $\cone(g_{2})=\cone(f_{n-2})\in \mathcal{X}$.

By repeating this process, we obtain a chain
\begin{equation}\label{1.3}
A\stackrel{g_{n}}{\longrightarrow}N_{n-1}\stackrel{g_{n-1}}{\longrightarrow}N_{n-2}\stackrel{}\cdots{\longrightarrow}N_{2}\stackrel{g_{2}}{\longrightarrow}N_{1}\stackrel{g_{1}}{\longrightarrow}B
\end{equation}
such that $\cone(g_{i})=\cone(f_{n-i})\in \mathcal{X}$ for $1\leq i\leq n$. Combining (\ref{1.3}) with (\ref{1.1}), we obtain an $\mathcal{X}$-filtration of $B$ with length $m+n$. Hence $l_{\mathcal{X}}(B)\leq m+n=l_{\mathcal{X}}(A)+l_{\mathcal{X}}(C)$, and $B\in\mathbf{Filt_{\mathscr{C}}(\mathcal{X})}$.
\end{proof}


In general, the equation in Lemma \ref{1} does not hold (see Example \ref{4.6}), but it does under certain conditions (see Lemma \ref{T}).
\begin{lemma}\label{2} Let $\mathcal{X}$ be a collection of objects in $\mathscr{C}$, then $\Filt_{\mathscr{C}}(\mathcal{X})$ is the smallest extension-closed subcategory in $\mathscr{C}$ containing $\mathcal{X}$.
\end{lemma}
\begin{proof} By Lemma \ref{1}, $\Filt_{\mathscr{C}}(\mathcal{X})$ is closed under extensions. The minimality can be followed by the induction on lengths of objects in $\Filt_{\mathscr{C}}(\mathcal{X})$.
\end{proof}

We have the following basic observation which will be used frequently in what follows.
\begin{lemma}\label{3} Let $\mathcal{X}$ be a collection of objects in $\mathscr{C}$, $M\in\Filt_{\mathscr{C}}(\mathcal{X})$ with $l_{\mathcal{X}}(M)=n$. Take an $\mathcal{X}$-filtration $0=X_{0}\stackrel{f_{0}}{\longrightarrow}X_{1}\stackrel{f_{1}}{\longrightarrow}X_{2}\stackrel{f_{2}}{\longrightarrow}\cdots\stackrel{f_{n-1}}{\longrightarrow}X_{n}=M$. Then the following statements hold.

$(1)$ $l_{\mathcal{X}}(X_{i})=i$ for $0\leq i\leq n$.

$(2)$ $l_{\mathcal{X}}(\cone(f_{j}f_{j-1}\cdots f_{i}))=j-i+1$ for $0\leq i\leq j\leq n-1$.
\end{lemma}
\begin{proof} $(1)$ Clearly, $l_{\mathcal{X}}(X_{i})=i$ holds for $i=0,1$. Assume the assertion is true for $i=k-1$. By Lemma \ref{1}, we obtain $l_{\mathcal{X}}(X_{k})\leq l_{\mathcal{X}}(X_{k-1})+l_{\mathcal{X}}(\cone(f_{k-1}))=k$. Suppose that $l_{\mathcal{X}}(X_{k})<k$, then we can obtain an $\mathcal{X}$-filtration of $M$ with length less than $n$, which contradicts with $l_{\mathcal{X}}(M)=n$. Hence, $l_{\mathcal{X}}(X_{k})=k$. By induction, we finish the proof.

$(2)$ We proceed the proof by induction on $s=j-i$. The case $s=0$ is clear since $l_{\mathcal{X}}(\cone(f_{i}))=1$. For any $0\leq i< j\leq n-1$,  $\rm (ET4)$ yields the following exact
commutative diagram
$$\xymatrix{
  X_{i} \ar@{=}[d] \ar[r] & X_{j}  \ar[d]^{f_{j}} \ar[r] & \cone(f_{j-1}f_{j-2}\cdots f_{i}) \ar[d]\\
  X_{i}  \ar[r] & X_{j+1} \ar[d] \ar[r] & \cone(f_{j}f_{j-1}\cdots f_{i})  \ar[d]\\
  & \cone(f_{j}) \ar@{=}[r] &  \cone(f_{j}) .}
$$
By induction, we obtain that
$$l_{\mathcal{X}}(\cone(f_{j}f_{j-1}\cdots f_{i}))\leq l_{\mathcal{X}}(\cone(f_{j-1}f_{j-2}\cdots f_{i}))+l_{\mathcal{X}}(f_{j})=j-i+1.$$
Set $l_{\mathcal{X}}(\cone(f_{j}f_{j-1}\cdots f_{i}))=m$, take an $\mathcal{X}$-filtration of $\cone(f_{j}f_{j-1}\cdots f_{i})$
$$0=Y_{0}\stackrel{}{\longrightarrow}Y_{1}\stackrel{}{\longrightarrow}Y_{2}{\longrightarrow}\cdots\stackrel{g_{m-1}}{\longrightarrow}Y_{m}=\cone(f_{j}f_{j-1}\cdots f_{i}).$$
By $\rm (ET4)^{op}$, we have the following exact commutative diagram
$$\xymatrix{
  X_{i} \ar@{=}[d] \ar[r] & N  \ar[d] \ar[r] & Y_{m-1}  \ar[d]^{g_{m-1}} \\
  X_{i}  \ar[r] & X_{j+1} \ar[d] \ar[r] &  Y_{m} \ar[d] \\
  &  \cone(g_{m-1})  \ar@{=}[r] & \cone(g_{m-1}). }
$$
Noting that $j+1=l_{\mathcal{X}}(X_{j+1})\leq l_{\mathcal{X}}(N)+1\leq i+m-1+1$, i.e., $m\geq j-i+1$, we obtain that $l_{\mathcal{X}}(\cone(f_{j}f_{j-1}\cdots f_{i}))=j-i+1$.
Therefore, we complete the proof.
\end{proof}

\section{Semibricks}
Recall that an object in an additive category $\mathcal{C}$ is called a {\em brick}, if its endomorphism ring is a division algebra. A set $\mathcal{X}$ of isoclasses of bricks in $\mathcal{C}$ is called a {\em semibrick} if $\Hom_{\mathcal{C}}(X_1,X_2)=0$ for any two non-isomorphic objects $X_1,X_2$ in $\mathcal{X}$.

Let us introduce the notions of simple objects and wide subcategories in an extriangulated category $\mathscr{C}$.
\begin{definition} Let $\mathscr{C}$ be an extriangulated category.

$(a)$ A morphism $f:A\rightarrow B $ in $\mathscr{C}$ is called {\em admissible} if there exists a deflation $h: A\rightarrow C$ and an inflation $g: C\rightarrow B$ in $\mathscr{C}$ such that $f=gh$.


$(b)$ A non-zero object $M$ in $\mathscr{C}$ is called a {\em simple} object if there does not exist an $\mathbb{E}$-triangle $A\stackrel{}{\longrightarrow}M\stackrel{}{\longrightarrow}B\stackrel{}\dashrightarrow$ in $\mathscr{C}$ such that $A,B\neq0$.

$(c)$ A set $\mathcal{X}$ of isoclasses of objects in $\mathscr{C}$ is called {\em simple} if $\mathcal{X}\subseteq \cim(\Filt_{\mathscr{C}}(\mathcal{X}))$, where and elsewhere we denote by $\cim(\mathcal {C})$ the collection of isoclasses of simple objects in an extriangulated category $\mathcal {C}$.

\end{definition}

\begin{definition}\label{wide} Let $\mathscr{C}$ be an extriangulated category. A subcategory $\mathcal{D}$ of $\mathscr{C}$ is  {\em wide} if the following conditions hold:

$(a)$ Every morphism $f$ in $\mathcal{D}$ is admissible.

$(b)$ $\mathcal{D}$ is closed under extensions.
\end{definition}
\begin{remark}
By \cite[Remark 2.18]{Na}, in Definition \ref{wide}, $\mathcal{D}$ is an extriangulated category, and the inclusion functor $i:\mathcal{D}\hookrightarrow\mathscr{C}$ is exact in the sense of \cite[Definition 2.31]{Ben}. Note that if $\mathscr{C}$ is an exact category, then the condition $(a)$ holds if and only if $\mathcal{D}$ is an abelian category (cf. \cite[Exercise 8.6]{Bu}). In this case, Definition \ref{wide} coincides with the usual wide subcategory.
\end{remark}
A subcategory $\mathcal{D}$ of an extriangulated category $\mathscr{C}$ is {\em length} if $\mathcal{D}=\Filt_{\mathcal{D}}(sim(\mathcal{D}))$.
Two morphisms $f: A\rightarrow B$ and $g:A'\rightarrow B'$ in $\mathscr{C}$ are said to be {\em isomorphic}, denoted by $f\simeq g$, if there are isomorphisms $x:A\rightarrow A'$ and $y:B\rightarrow B'$ in $\mathscr{C}$ such that $yf=gx$.
\begin{lemma}\label{A} Suppose that $f\simeq g$, then $f$ is an inflation (resp. deflation) if and only if $g$ is an inflation (resp. deflation).
\end{lemma}
\begin{proof} It is easily proved by \cite[Proposition 3.7]{Na}.
\end{proof}

\begin{lemma}\label{4}  Let $\mathscr{C}$ be an extriangulated category. Let $\mathcal{X}$ be a semibrick in $\mathscr{C}$ and $f:X\rightarrow M$ be a morphism in $\Filt_{\mathscr{C}}(\mathcal{X})$ with $X\in\mathcal{X}$. Then $f=0$ or $f$ is an inflation such that $l_{\mathcal{X}}(\cone(f))=l_{\mathcal{X}}(M)-1$.
\end{lemma}
\begin{proof} We proceed the proof by induction on $l_{\mathcal{X}}(M)=n$. The cases of $n=0,1$ are trivial. Let us firstly deal with the case of $n=2$. Consider the following diagram
\begin{equation}\label{3.1}\xymatrix{
   & X\ar[d]_{f} \ar[dr]^{xf} & &\\
  Y_{1} \ar[r]^{y} & M \ar[r]^{x} & N \ar@{-->}[r]^{\delta}& }\end{equation}
with $Y_{1},N\in\mathcal{X}$. Assume that $f$ is non-zero. Observe that $xf$ is zero or $xf$ is an isomorphism since $X,N\in\mathcal{X}$. For the former, there exists a morphism $z:X\rightarrow Y_{1}$ such that $f=yz$. It should be noted that $z$ is an isomorphism since $f$ is non-zero. By \cite[Proposition 3.7]{Na}, we know that $X\stackrel{yz}{\longrightarrow}M\stackrel{x}{\longrightarrow}N\stackrel{(z^{-1})^{\ast}\delta}\dashrightarrow$ is an $\mathbb{E}$-triangle. Thus, $f=yz$ is an inflation, and $l_{\mathcal{X}}(\cone(f))=l_{\mathcal{X}}(N)=1=l_{\mathcal{X}}(M)-1$. For the latter, then the $\mathbb{E}$-triangle in (\ref{3.1}) is splitting. Therefore there exists an isomorphism $\footnotesize\begin{pmatrix} a \\x \end{pmatrix}:M\rightarrow Y_{1}\oplus N$. We can verify directly that $f\simeq\footnotesize\begin{pmatrix} af \\xf \end{pmatrix}\simeq\footnotesize\begin{pmatrix} af \\1 \end{pmatrix}\simeq\footnotesize\begin{pmatrix} 0 \\1 \end{pmatrix}: X\rightarrow Y_{1}\oplus X$ as morphisms. By Lemma \ref{A}, $f$ is an inflation, and $l_{\mathcal{X}}(\cone(f))=l_{\mathcal{X}}(Y_1)=1=l_{\mathcal{X}}(M)-1$.

Now we consider the case of $n\geq2$. Consider the following commutative diagram
\begin{equation}\label{3.2}\xymatrix{
   & X\ar[d]_{f} \ar[dr]^{xf} & &\\
  Y_{n-1} \ar[r]^{y} & M \ar[r]^{x} & N \ar@{-->}[r]^{\delta}& }\end{equation}
with $l_{\mathcal{X}}(Y_{n-1})=n-1$ and $l_{\mathcal{X}}(N)=1$. If $xf$ is an isomorphism, the assertion can be proved by repeating the latter process in the case of $n=2$. If $xf=0$, there exists $z:X\rightarrow Y_{n-1}$ such that $f=yz$. By induction, $z$ is an inflation such that $l_{\mathcal{X}}(\cone(z))=n-2$. Applying $\rm(ET4)$, we have the following exact commutative diagram
$$\xymatrix{
  X \ar@{=}[d] \ar[r]^{z} & Y_{n-1} \ar[d]^{y} \ar[r] & \cone(z) \ar[d] \\
  X \ar[r]^{f} & M \ar[d] \ar[r] &\cone(f) \ar[d] \\
   & N\ar@{=}[r] & N .  }
$$
Thus, $f$ is an inflation.
By Lemma \ref{1}, $l_{\mathcal{X}}(\cone(f))\leq l_{\mathcal{X}}(\cone(z))+l_{\mathcal{X}}(N)=n-1$. On the other hand, $n=l_{\mathcal{X}}(M)\leq 1+l_{\mathcal{X}}(\cone(f))$, i.e., $l_{\mathcal{X}}(\cone(f))\geq n-1$. Hence, $l_{\mathcal{X}}(\cone(f))= n-1=l_{\mathcal{X}}(M)-1$. This finishes the proof.
\end{proof}

\begin{corollary}\label{3.5} Let $\mathscr{C}$ be an extriangulated category. Let $\mathcal{X}$ be a semibrick in $\mathscr{C}$ and $f:M\rightarrow X$ be a morphism in $\Filt_{\mathscr{C}}(\mathcal{X})$ with $X\in\mathcal{X}$. Then $f=0$ or $f$ is a deflation such that $l_{\mathcal{X}}(\cocone(f))=l_{\mathcal{X}}(M)-1$.
\end{corollary}
\begin{proof}It is proved dually by using Remark \ref{2.6} and Lemma \ref{A}.\end{proof}

\section{Semibricks and wide subcategories}
In this section, let $\mathscr{C}$ be always an extriangulated category.
Let us state the first main result in this paper as the following
\begin{theorem}\label{main}
Let $\mathscr{C}$ be an extriangulated category. The assignments $\mathcal{X}\mapsto\Filt_{\mathscr{C}}(\mathcal{X})$ and $\mathcal{D}\mapsto\cim(\mathcal{D})$ give one-to-one
correspondence between the following two classes.

$(1)$ The class of simple semibricks $\mathcal{X}$ in $\mathscr{C}$.

$(2)$ The class of length wide subcategories $\mathcal{D}$ of $\mathscr{C}$.
\end{theorem}

Before proving Theorem \ref{main}, we need some preparations.

\begin{lemma}\label{3.8}Let $\mathcal{X}$ be a semibrick in $\mathscr{C}$, and $f:M\rightarrow N$ be a nonzero morphism in $\Filt_{\mathscr{C}}(\mathcal{X})$.

$(1)$ If $l_{\mathcal{X}}(M)=l_{\mathcal{X}}(N)=2$, then either $f$ factors through some $X\in\mathcal{X}$ or $f$ is an isomorphism.

$(2)$ If $l_{\mathcal{X}}(M)=2$, then either $f$ factors through some $X\in\mathcal{X}$ or $f=f_{2}f_{1}$ with $f_{1}$ being an isomorphism and $f_{2}$ being an inflation.

$(2')$ If $l_{\mathcal{X}}(N)=2$, then either $f$ factors through some $X\in\mathcal{X}$ or $f=g_{2}g_{1}$ with $g_{1}$ being a deflation and $g_{2}$ being an isomorphism.

$(3)$ If $l_{\mathcal{X}}(M)=l_{\mathcal{X}}(N)=3$, then $f$ factors through some $X\in\mathcal{X}$, or $f$ is an isomorphism, or $f=f_{3}f_{2}f_{1}$, where $f_{1}:M\rightarrow W_{1}$ is a deflation, $f_{2}:W_{1}\rightarrow W_{2}$ is an isomorphism, $f_{3}:W_{2}\rightarrow N$ is an inflation, and $l_{\mathcal{X}}(W_{2})<l_{\mathcal{X}}(N)$.
\end{lemma}
\begin{proof} Let $l_{\mathcal{X}}(M)=m$ and $l_{\mathcal{X}}(N)=n$. Take an $\mathbb{E}$-triangle of the from $X_{1}\stackrel{a}{\longrightarrow}M\stackrel{b}{\longrightarrow}X_{2}\stackrel{}\dashrightarrow$ with $X_{1}\in \mathcal{X}$ and $l_{\mathcal{X}}(X_{2})=m-1$.

$(1)$~If $fa=0$, then $f$ factors through $X_{2}$. Otherwise, by Lemma \ref{4}, we have the following commutative diagram
\begin{equation}\label{3.10}
\xymatrix{
  X_{1} \ar@{=}[d] \ar[r]^-{a} & M\ar[d]_-{f} \ar[r]^-{b} & X_{2} \ar[d]_-{t} \ar@{-->}[r]^-{t^{\ast}\delta}  &  \\
  X_{1} \ar[r]^-{fa} & N \ar[r]^-{c} & Y \ar@{-->}[r]^-{\delta}&  }
\end{equation}
with $l_{\mathcal{X}}(Y)=n-1$. If $l_{\mathcal{X}}(M)=l_{\mathcal{X}}(N)=2$, then $l_{\mathcal{X}}(Y)=l_{\mathcal{X}}(X_{2})=1$, i.e., $X_2,Y\in\mathcal{X}$. It follows that $t=0$ or $t$ is an isomorphism. If $t=0$, it implies that $f$
factors through $X_{1}$; if $t$ is an isomorphism, it implies that $f$ is an isomorphism.

In what follows, we always assume that $fa\neq0$, since $fa=0$ implies that $f$ factors through $X_2\in\mathcal{X}$. In this case, we still have the commutative diagram (\ref{3.10}).

$(2)$~If $l_{\mathcal{X}}(M)=2$, keep the notation as $(1)$, then by Lemma \ref{4}, we know that $t=0$ or $t$ is an inflation since $X_2\in\mathcal{X}$. If $t=0$, it implies that $f$
factors through $X_{1}$; if $t$ is an inflation, consider the following commutative diagram by $\rm (ET4)^{op}$
\begin{equation}\label{3.11}
\xymatrix{
    M \ar@/_/[ddr]_{f} \ar@/^/[drr]^{b}
    \ar@{.>}[dr]|-{l}  &  &   &  \\
  X_{1} \ar@{=}[d] \ar[r] & M' \ar[d]^-{f'} \ar[r] & X_{2} \ar[d]_-{t} \ar@{-->}[r]^-{t^{\ast}\delta}  &  \\
  X_{1} \ar[r]^-{fa} & N\ar[r]^-{c} & Y \ar[r]^-{\delta} &    }
\end{equation}
with $f'$ being an inflation. Moreover, $l_{\mathcal{X}}(M')=2$ by Lemma \ref{4}. By Lemma \ref{2.5}, there exists a morphism $l$ such that $f=f'l$. It has proved in (1) that $l$ factors through some object in $\mathcal{X}$ or $l$ is an isomorphism. Hence, we complete the proof of $(2)$. The statement $(2')$ can be proved dually.

$(3)$~If $l_{\mathcal{X}}(M)=l_{\mathcal{X}}(N)=3$, keep the notation as $(1)$, then $l_{\mathcal{X}}(Y)=l_{\mathcal{X}}(X_{2})=2$. By $(1)$, it follows that $t$ is an isomorphism or $t=e_{2}e_{1}$ for some $e_{1}:X_{2}\rightarrow X'$ and $e_{2}:X'\rightarrow Y$ with $X'\in\mathcal{X}$. If $t$ is an isomorphism, then $f$ is an isomorphism; for the latter case, repeating the process in (\ref{3.11}) through replacing $t$ by $e_{2}$, we have the following commutative diagram
\begin{equation}\label{chongfu}
\xymatrix{
    M \ar@/_/[ddr]_{f} \ar@/^/[drr]^{e_1b}
    \ar@{.>}[dr]|-{l'}  &  &   &  \\
  X_{1} \ar@{=}[d] \ar[r] & M'' \ar[d]^-{f''} \ar[r] & X' \ar[d]_-{e_2} \ar@{-->}[r]^-{t^{\ast}\delta}  &  \\
  X_{1} \ar[r]^-{fa} & N\ar[r]^-{c} & Y \ar[r]^-{\delta} &    }
\end{equation}
with $f=f''l'$. Note that $l_{\mathcal{X}}(M'')=2$ and $f''$ is an inflation. By $(2')$, we finish the proof.
\end{proof}

\begin{proposition}\label{4.5} Let $\mathcal{X}$ be a semibrick in $\mathscr{C}$, then $\Filt_{\mathscr{C}}(\mathcal{X})$ is a  wide subcategory. In addition, if $\mathcal{X}$ is simple, then $\Filt_{\mathscr{C}}(\mathcal{X})$ is a length wide subcategory.
\end{proposition}
\begin{proof}  Let $f:M\rightarrow N$ be an arbitrary morphism in $\Filt_{\mathscr{C}}(\mathcal{X})$. Without loss of generality, we assume that $f$ is nonzero and not an isomorphism. By Lemma \ref{3.8}, $f$ is admissible if $l_{\mathcal{X}}(M)\leq2$ or $l_{\mathcal{X}}(N)\leq2$.

Assume that $l_{\mathcal{X}}(M), l_{\mathcal{X}}(N)>2$. Then we claim that $f$ factors through some object in $\mathcal{X}$ or $f=f_{3}f_{2}f_{1}$ such that $f_{1}:M\rightarrow W_{1}$ is a deflation, $f_{2}:W_{1}\rightarrow W_{2}$ is an isomorphism and $f_{3}:W_{2}\rightarrow N$ is an inflation with $l_{\mathcal{X}}(W_{2})<l_{\mathcal{X}}(N)$.
We proceed the proof of this claim by induction on $l_{\mathcal{X}}(M)+l_{\mathcal{X}}(N)$. The case of $l_{\mathcal{X}}(M)=l_{\mathcal{X}}(N)=3$ follows from  Lemma \ref{3.8}.  By Lemma \ref{3}, we obtain the following diagram
\begin{equation}
\xymatrix{
  X \ar[r]^-{g}\ar[r]\ar[dr]_-{fg} &M \ar[d]_-{f} \ar[r]^-{h} & \cone(g)\ar@{-->}[r]\ar@{.>}^-{e}[dl] &\\
   & N &  &  }
\end{equation}
with $X\in \mathcal{X}$ and  $l_{\mathcal{X}}(\cone(g))=l_{\mathcal{X}}(M)-1$. If $fg=0$, then $f=eh$ for some morphism $e:\cone(g)\rightarrow N$. Since $l_{\mathcal{X}}(\cone(g))+l_{\mathcal{X}}(N)<l_{\mathcal{X}}(M)+l_{\mathcal{X}}(N)$, by induction, the claim holds for $e$; if $fg$ is nonzero, by Lemma \ref{4}, we have the following commutative diagram
\begin{equation}
\xymatrix{
  X \ar@{=}[d] \ar[r]^-{g} & M\ar[d]_-{f} \ar[r]^-{h} & \cone(g)\ar@/^0.8pc/[d]|{X'} \ar[d]_-{t} \ar@{-->}[r]^-{t^{\ast}\delta}  &  \\
  X \ar[r]^-{fg} & N \ar[r]^-{s} & Y \ar@{-->}[r]^-{\delta}&   . }
\end{equation}
Since $l_{\mathcal{X}}(\cone(g))+l_{\mathcal{X}}(Y)=l_{\mathcal{X}}(M)+l_{\mathcal{X}}(N)-2$, by induction, the claim holds for $t$.

Suppose that $t$ factors through some $X'\in\mathcal{X}$, i.e., $t=t_{2}t_{1}$ for $t_{1}:\cone(g)\rightarrow X'$ and $t_{2}:X'\rightarrow Y$. If $t=0$, then $f$ factors through $X\in\mathcal{X}$; if $t$ is nonzero, then $t_{2}$ is an inflation. Applying $\rm (ET4)^{op}$ yields the following  commutative diagram
\begin{equation}\label{3.14}\xymatrix{
    M \ar@/_/[ddr]_{f} \ar@/^/[drr]^{t_{1}h}
    \ar@{.>}[dr]|-{l}  &  &   &  \\
  X \ar@{=}[d] \ar[r] & M' \ar[d]_-{f'} \ar[r]^-{h'} & X' \ar[d]_-{t_{2}} \ar@{-->}[r]^-{t_{2}^{\ast}\delta}  &  \\
  X \ar[r]^-{fg} & N\ar[r]^-{s} & Y \ar@{-->}[r]^-{\delta} &   }
\end{equation}
with $l_{\mathcal{X}}(M')=2$ and $f'$ being an inflation. By Lemma \ref{2.5}, there exists a morphism $l$ such that $f=f'l$. By Lemma \ref{3.8}, $l$ factors as a composition of an isomorphism with a deflation.

Suppose that $t=e_{3}e_{2}e_{1}$ such that $e_{1}:\cone(g)\rightarrow Y_{1}$ is a deflation, $e_{2}:Y_{1}\rightarrow Y_{2}$ is an isomorphism and $e_{3}:Y_{2}\rightarrow Y$ is an inflation with $l_{\mathcal{X}}(Y_{2})<l_{\mathcal{X}}(Y)$. Repeating the process in (\ref{3.14}) through replacing $t_{2}$ by $e_{3}$. Then there exists a morphism $l':M\rightarrow M''$ and $f'':M''\rightarrow N$ such that $f=f''l'$, where $l_{\mathcal{X}}(M'')=l_{\mathcal{X}}(Y_{2})+1$ and $f''$ is an inflation. Note that $l_{\mathcal{X}}(M)+l_{\mathcal{X}}(M'')=l_{\mathcal{X}}(M)+l_{\mathcal{X}}(Y_2)+1<l_{\mathcal{X}}(M)+l_{\mathcal{X}}(Y)+1=l_{\mathcal{X}}(M)+l_{\mathcal{X}}(N)$. By induction, the claim holds for $l'$. Then it follows that the claim holds for $f$. Thus we complete the proof of the claim.
Hence, every morphism in $\Filt_{\mathscr{C}}(\mathcal{X})$ is admissible. By Lemma \ref{1}, we obtain that it is a wide subcategory.

For each simple object $S\in\Filt_{\mathscr{C}}(\mathcal{X})$, its length $l_{\mathcal{X}}(S)$ is equal one, it follows that $\cim(\Filt_{\mathscr{C}}(\mathcal{X}))\subseteq \mathcal{X}$. If $\mathcal{X}$ is simple, then $\cim(\Filt_{\mathscr{C}}(\mathcal{X}))=\mathcal{X}$, i.e., $\Filt_{\mathscr{C}}(\mathcal{X})$ is a length wide subcategory. \end{proof}

Now we are in the position to prove Theorem \ref{main}.

\textbf{{Proof of Theorem \ref{main}.}} Proposition \ref{4.5} implies that $\Filt_{\mathscr{C}}(\mathcal{X})$ is a length wide subcategory and $\cim(\Filt_{\mathscr{C}}(\mathcal{X}))=\mathcal{X}$ if $\mathcal{X}$ is a simple semibrick. Let $\mathcal{D}$ be a length wide subcategory, then $\cim(\mathcal{D})$ is a simple semibrick since every morphism in $\mathcal{D}$ is admissible. Observe that $\mathcal{D}=\Filt_{\mathcal{D}}(\cim(\mathcal{D}))\subseteq\Filt_{\mathscr{C}}(\cim(\mathcal{D})$, by Lemma \ref{2}, we have that $\Filt_{\mathscr{C}}(\cim(\mathcal{D}))\subseteq \mathcal{D}$. This finishes the proof.\fin\\

We finish this section with a straightforward example illustrating Theorem \ref{main}.
\begin{example}\label{4.6} Consider the path algebra $A$ of the quiver $1\longleftarrow2\longleftarrow3$.
The Auslander--Reiten quiver is given by
$$\xymatrix{
   &  & P_3 \ar[dr] & & \\
  & P_2\ar[ur]\ar[dr] \ar@{--}[r]& \ar@{--}[r] & I_2 \ar[dr] &  \\
  S_1 \ar[ur]\ar@{--}[r] &  \ar@{--}[r]& S_2 \ar@{--}[r] \ar[ur]&  \ar@{--}[r] &S_3~.   }$$
Let $\mathcal{D}=\add\{S_2\oplus I_2\oplus S_3\}$, $\mathcal{X}=\{S_2,S_3\}$ and $\mathcal{Y}=\{S_2,S_3,I_2\}$. Then $\mathcal{D}=\Filt_{\mod A}(\mathcal{X})$ is a length wide subcategory and $\cim(\mathcal{D})=\mathcal{X}$ is a simple semibrick. In addition, there is a short exact sequence
$0{\longrightarrow}S_2{\longrightarrow}I_2{\longrightarrow}S_3{\longrightarrow}0$ in $\Filt_{\mod \Lambda}(\mathcal{Y})$ with $l_{\mathcal{Y}}(I_2)=1<2$, which shows that the equation in Lemma \ref{1} does not hold in general.
\end{example}

\section{Semibricks and cotorsion pairs}
Let $\mathscr{C}$ be an extriangulated category.
Let $\mathcal{U}$, $\mathcal{V}\subseteq\mathscr{C}$ be a pair of subcategories which are closed under direct summands. Recall that the pair $(\mathcal{U},\mathcal{V})$ is called a {\em cotorsion pair} in $\mathscr{C}$ if it satisfies the following conditions:

$(a)$~$\mathbb{E}(\mathcal{U},\mathcal{V})=0.$

$(b)$~For any $C\in\mathscr{C}$, there exists a conflation $V\longrightarrow U\longrightarrow C$ such that $V\in\mathcal{V}$, $U\in\mathcal{U}$.

$(c)$~For any $C\in\mathscr{C}$, there exists a conflation $C\longrightarrow V'\longrightarrow U'$ such that $V'\in\mathcal{V}$, $U'\in\mathcal{U}$.

\begin{remark}
Let $(\mathcal{U}$, $\mathcal{V})$ be a cotorsion pair in an extriangulated category $\mathscr{C}$. Then
\begin{itemize}
\item $M\in\mathcal{U}$ if and only if $\mathbb{E}(M,\mathcal{V})=0$;
\item $N\in\mathcal{V}$ if and only if $\mathbb{E}(\mathcal{U},N)=0$;
\item $\mathcal{U}$ and $\mathcal{V}$ are extension-closed;
\item $\mathcal{U}$ is contravariantly finite and $\mathcal{V}$ is covariantly finite in $\mathscr{C}$;
\item $\mathcal{P}\subseteq\mathcal{U}$ and $\mathcal{I}\subseteq\mathcal{V}$.
\end{itemize}
\end{remark}

\begin{lemma}\label{T} Let $\mathscr{C}$ be an extriangulated category and $\mathcal{X}$ be a semibrick in $\mathscr{C}$.

$(1)$ $\Filt_{\mathscr{C}}(\mathcal{X})$ is closed under direct summands in $\mathscr{C}$.

$(2)$ For any object $X\in\Filt_{\mathscr{C}}(\mathcal{X})$, if $X=A\oplus B$, then $l_{\mathcal{X}}(X)=l_{\mathcal{X}}(A)+l_{\mathcal{X}}(B)$.
\end{lemma}
\begin{proof} We proceed the proofs of $(1)$ and $(2)$ by induction on the length $l_{\mathcal{X}}(X)=n$ of an object $X\in\Filt_{\mathscr{C}}(\mathcal{X})$. The case $n=0$ is trivial. If $n=1$, the assertions are also clear since each brick is indecomposable. For $n>1$,  without loss of generality, we assume that $X$ is decomposable, and let $X=A\oplus B$ with $A,B\neq 0$. Consider the following diagram
\begin{equation}
\xymatrix{ & B \ar^-{0\choose1}[d]&  & \\
  X_{n-1}  \ar[r]^-{a\choose b} & A\oplus B  \ar[r]^-{(c~d)} & X_{1}  \ar@{-->}[r]&
   }
\end{equation}
with $l_{\mathcal{X}}(X_{n-1})=n-1$ and $X_{1}\in \mathcal{X}$.

If $d=0$, then $B$ is a direct summand of $X_{n-1}$. By induction, $B\in\Filt_{\mathscr{C}}(\mathcal{X})$ and $l_{\mathcal{X}}(B)\leq n-1$. Applying $\rm(ET4)$ together with Lemma \ref{3}, we have the following exact commutative diagram
$$\xymatrix{
  Y_{1} \ar@{=}[d] \ar[r]^-{z} & B \ar[d]^-{0\choose1} \ar[r] & B' \ar[d] \\
  Y_{1}\ar[r]^-{f} & A\oplus B \ar^-{(1~0)}[d] \ar[r] &H \ar[d] \\
   & A\ar@{=}[r] \ar@{-->}[d]^{0} & A \ar@{-->}[d]^{0} \\
  && }
$$
with $Y_{1}\in \mathcal{X}$ and $l_{\mathcal{X}}(B')=l_{\mathcal{X}}(B)-1$. If $f=0$, then $z=0$ and thus $B$ is a direct summand of $B'$. By induction, $l_{\mathcal{X}}(B)\leq l_{\mathcal{X}}(B')=l_{\mathcal{X}}(B)-1$. This is a contradiction. Hence, $f\neq0$. By Lemma \ref{4}, we have that $l_{\mathcal{X}}(H)=l_{\mathcal{X}}(A\oplus B)-l_{\mathcal{X}}(Y_{1})=n-1$.
By induction, $A\in\Filt_{\mathscr{C}}(\mathcal{X})$ and $l_{\mathcal{X}}(H)=l_{\mathcal{X}}(A)+l_{\mathcal{X}}(B')$. Thus,
$$l_{\mathcal{X}}(A\oplus B)=l_{\mathcal{X}}(Y_{1})+l_{\mathcal{X}}(H)=l_{\mathcal{X}}(Y_{1})+l_{\mathcal{X}}(B')+l_{\mathcal{X}}(A)=l_{\mathcal{X}}(B)+l_{\mathcal{X}}(A).$$

If $d\neq0$, by Corollary \ref{3.5}, we have an $\mathbb{E}$-triangle
$$Y\stackrel{h}{\longrightarrow}A\oplus B\stackrel{(0~d)}{\longrightarrow}X_{1}\stackrel{\delta}\dashrightarrow$$
with $l_{\mathcal{X}}(Y)=n-1$. Then the inclusion $A\hookrightarrow A\oplus B$ factors through $h$, and thus $A$ is a direct summand of $Y$. By induction, $A\in\Filt_{\mathscr{C}}(\mathcal{X})$ and $l_{\mathcal{X}}(A)\leq n-1$. Then we can complete the proof by repeating the process in the case of $d=0$.\end{proof}

Recall that a triangulated subcategory $\mathcal{D}$ of a triangulated category $\mathscr{C}$ is called a {\em thick subcategory} if $\mathcal{D}$ is closed under direct summands in $\mathscr{C}$.
\begin{corollary} Let $\mathscr{C}$ be a triangulated category with the suspension functor $[1]$. Let $\mathcal{X}$ be a semibrick such that $\mathcal{X}[1],\mathcal{X}[-1]\subseteq\Filt_{\mathscr{C}}(\mathcal{X})$. Then $\Filt_{\mathscr{C}}(\mathcal{X})$ is a thick subcategory of $\mathscr{C}$.
\end{corollary}
\begin{proof} By Lemma \ref{2} and Lemma \ref{T}, we only need to prove $\Filt_{\mathscr{C}}(\mathcal{X})$ is closed under $[1]$ and $[-1]$. We will proceed the proof by induction on the lengths of the objects in $\Filt_{\mathscr{C}}(\mathcal{X})$. Let $M\in\Filt_{\mathscr{C}}(\mathcal{X})$ with $l_{\mathcal{X}}(M)=n$. If $n=1$, $M[1]\in\Filt_{\mathscr{C}}(\mathcal{X})$, since $\mathcal{X}[1]\subseteq\Filt_{\mathscr{C}}(\mathcal{X})$; if $n>1$, by Lemma \ref{3}, there exists a triangle $X\stackrel{}{\longrightarrow}M\stackrel{}{\longrightarrow}N\stackrel{}\rightarrow X[1]$ with $X\in \mathcal{X}$ and $l_{\mathcal{X}}(N)=n-1$. Note that $X[1]\in\Filt_{\mathscr{C}}(\mathcal{X})$, and by induction, $N[1]\in \Filt_{\mathscr{C}}(\mathcal{X})$. Then it follows that $M[1]\in\Filt_{\mathscr{C}}(\mathcal{X})$, since $\Filt_{\mathscr{C}}(\mathcal{X})$ is closed under extensions. Similarly, using $\mathcal{X}[-1]\subseteq\Filt_{\mathscr{C}}(\mathcal{X})$, we can prove $\Filt_{\mathscr{C}}(\mathcal{X})$ is closed under $[-1]$. Hence, $\Filt_{\mathscr{C}}(\mathcal{X})$ is a thick subcategory of $\mathscr{C}$.
\end{proof}

Let $\mathscr{C}$ be an extriangulated category and $\mathcal{X}$ be a semibrick in $\mathscr{C}$, then $\mathcal{T}=\Filt_{\mathscr{C}}(\mathcal{X})$ is an extriangualted category. Given a subcategory $\mathcal{D}$ of $\mathcal{T}$, we denote by $\mathcal{S}_{\mathcal{D}}$ the subset of $\mathcal{X}$ such that $\Filt_{\mathcal{T}}(\mathcal{S_{\mathcal{D}}})$ is the smallest filtration subcategory containing $\mathcal{D}$ in $\mathcal{T}$.

\begin{lemma}\label{B} Let $\mathcal{X}$ be a semibrick in an extriangulated category $\mathscr{C}$ and $\mathcal{T}=\Filt_{\mathscr{C}}(\mathcal{X})$.

$(1)$ For any subsets $S',S\subseteq \mathcal{X}$, $\Filt_{\mathscr{C}}(S')\subseteq\Filt_{\mathscr{C}}(S)$ if and only if $S'\subseteq S$. In particular, $\Filt_{\mathscr{C}}(S')=\Filt_{\mathscr{C}}(S)$ if and only if $S'= S$.

$(2)$ $\Filt_{\mathcal{T}}(\mathcal{S}_{\mathcal{D}})=\mathcal{D}$  if and only if $\mathcal{D}$ is a filtration subcategory of $\mathcal{T}$.
\end{lemma}
\begin{proof} (1) We only need to prove the necessity of the first statement. For any $X\in S'$, then $X\in\Filt_{\mathscr{C}}(S')\subseteq\Filt_{\mathscr{C}}(S)$ with $l_{\mathcal{S}}(X)=n$. By Lemma \ref{3}, there exists an $\mathbb{E}$-triangle $X_{1}\stackrel{x}{\longrightarrow}X\stackrel{}{\longrightarrow}X_{2}\stackrel{}\dashrightarrow$ with  $X_{1}\in \mathcal{S}$ and $l_{\mathcal{S}}(X_{2})=n-1$. Note that $x$ is an isomorphism or zero. For the former, we obtain that $X\in S$; for the latter, we get that $X$ is a direct summand of $X_{2}$, by Lemma \ref{T}, it follows that $n=l_{\mathcal{S}}(X)\leq l_{\mathcal{S}}(X_{2})=n-1$, this is a contradiction. Hence, $S'\subseteq S$.

(2) We only need to prove the sufficiency. Suppose that $\mathcal{D}$ is a filtration subcategory, i.e., there exists a subset $\mathcal{Y}$ of $\mathcal{X}$ such that $\mathcal{D}=\Filt_{\mathcal{T}}(\mathcal{Y})$. By definition, $\Filt_{\mathcal{T}}(\mathcal{Y})=\mathcal{D}\subseteq\Filt_{\mathcal{T}}(\mathcal{S}_{\mathcal{D}})$. Moreover, by the minimality of $\Filt_{\mathcal{T}}(\mathcal{S}_{\mathcal{D}})$, we also have that $\Filt_{\mathcal{T}}(\mathcal{S}_{\mathcal{D}})\subseteq\Filt_{\mathcal{T}}(\mathcal{Y})$. Thus, $\Filt_{\mathcal{T}}(\mathcal{Y})=\Filt_{\mathcal{T}}(\mathcal{S}_{\mathcal{D}})$, that is, $\Filt_{\mathcal{T}}(\mathcal{S}_{\mathcal{D}})=\mathcal{D}$.
\end{proof}

\begin{proposition}\label{finite}
Let $\mathcal{X}$ be a semibrick in an extriangulated category $\mathscr{C}$ and $\mathcal{T}=\Filt_{\mathscr{C}}(\mathcal{X})$. Then for any subset $\mathcal{S}\subseteq\mathcal{X}$, we have the following

$(1)$ For any $M\in \mathcal{T}$, there exists an $\mathbb{E}$-triangle $N\stackrel{x}{\longrightarrow}M\stackrel{}{\longrightarrow}P\stackrel{}\dashrightarrow$ with $N\in\Filt_{\mathcal{T}}(\mathcal{S})$ and $P\in\mathcal{S}^{\perp}$.

$(2)$ For any $M\in \mathcal{T}$, there exists an $\mathbb{E}$-triangle $U\stackrel{}{\longrightarrow}M\stackrel{y}{\longrightarrow}V\stackrel{}\dashrightarrow$ with $U\in^{\perp}\mathcal{S}$ and $V\in\Filt_{\mathcal{T}}(\mathcal{S})$.

$(3)$ $\Filt_{\mathcal{T}}(\mathcal{S})$ is functorially finite in $\mathcal{T}$.

$(4)$ Assume that $\mathcal{T}$ has enough projectives and enough
injectives. If $\mathcal{S}_{\mathcal{P}}\subseteq\mathcal{S}$, then $(\Filt_{\mathcal{T}}(\mathcal{S}),\mathcal{S}^{\perp_{1}})$ is a cotorsion pair in $\mathcal{T}$. Dually, if $\mathcal{S}_{\mathcal{I}}\subseteq\mathcal{S}$, then $(^{\perp_{1}}\mathcal{S},\Filt_{\mathcal{T}}(\mathcal{S}))$ is a cotorsion pair in $\mathcal{T}$.
\end{proposition}
\begin{proof} $(1)$ If $M\in\mathcal{S}^{\perp}$, then we use the $\mathbb{E}$-triangle $0\stackrel{}{\longrightarrow}M\stackrel{}{\longrightarrow}M\stackrel{0}\dashrightarrow$. If $M\notin\mathcal{S}^{\perp}$, we proceed the proof by induction on $l_{\mathcal{X}}(M)=n$. If $n=1$, i.e., $M\in\mathcal{X}$. Since $M\notin\mathcal{S}^{\perp}$, there exists an object $S\in\mathcal{S}$ such that $\Hom(S,M)\neq0$, and then $M\cong S$, that is, $M\in\mathcal{S}$. Thus, we take the $\mathbb{E}$-triangle $M\stackrel{}{\longrightarrow}M\stackrel{}{\longrightarrow}0\stackrel{0}\dashrightarrow$, which is the desired. For $n>1$, since $M\notin\mathcal{S}^{\perp}$, there exists a nonzero morphism $f:S\rightarrow M$ for some object $S\in\mathcal{S}$, by Lemma \ref{4}, we have an $\mathbb{E}$-triangle $S\stackrel{f}{\longrightarrow}M\stackrel{}{\longrightarrow}H\stackrel{}\dashrightarrow$ with $l_{\mathcal{X}}(H)=n-1$. By induction, for $H$, we have an $\mathbb{E}$-triangle $H'\stackrel{}{\longrightarrow}H\stackrel{}{\longrightarrow}P\stackrel{}\dashrightarrow$ with $H'\in\Filt_{\mathcal{T}}(\mathcal{S})$ and $P\in\mathcal{S}^{\perp}$. Thus we have the following exact commutative diagram by $\rm (ET4)^{op}$
\begin{equation}\label{4.2}
\xymatrix{
   S\ar@{=}[d]\ar[r] & N \ar[d] \ar[r] & H' \ar[d] \\
  S \ar[r] & M\ar[d] \ar[r] & H \ar[d] \\
   & P \ar@{=}[r] & P~.    }
\end{equation}
Note that $N\in\Filt_{\mathcal{T}}(\mathcal{S})$ since $S$, $H'\in\Filt_{\mathcal{T}}(\mathcal{S})$. Then the second column in (\ref{4.2}) gives the desired $\mathbb{E}$-triangle.

$(2)$ It is similar to (1).

$(3)$ It is easy to see that $x:N\rightarrow M$ in (1) is a right $\Filt_{\mathcal{T}}(\mathcal{S})$-approximation and $y:M\rightarrow V$ in (2) is a left $\Filt_{\mathcal{T}}(\mathcal{S})$-approximation. Hence $\Filt_{\mathcal{T}}(\mathcal{S})$ is functorially finite in $\mathcal{T}$.

$(4)$ By Remark \ref{2.6}, we obtain that $\mathbb{E}(\Filt_{\mathcal{T}}(\mathcal{S}),\mathcal{S}^{\perp_{1}})=0$, and $\Filt_{\mathcal{T}}(\mathcal{S})$ is closed under direct summands by Lemma $\ref{T}$. As proved in (3), $\Filt_{\mathcal{T}}(\mathcal{S})$ is functorially finite. It follows that $(\Filt_{\mathcal{T}}(\mathcal{S}),\mathcal{S}^{\perp_{1}})$ is a cotorsion pair in $\mathcal{T}$  by \cite[Proposition 3.4]{Ch}. It is proved dually that $(^{\perp_{1}}\mathcal{S},\Filt_{\mathcal{T}}(\mathcal{S}))$ is also a cotorsion pair.
\end{proof}


Now we can give a relation between cotorsion pairs and semibricks in the following
\begin{theorem}\label{H} Let $\mathscr{C}$ be an extriangulated category and $\mathcal{X}$ be a semibrick in $\mathscr{C}$. Assume that $\mathcal{T}=\Filt_{\mathscr{C}}(\mathcal{X})$ has enough projectives and enough
injectives. The assignments $\mathcal{U}\mapsto\mathcal{S}_{\mathcal{U}}$
and $\mathcal{S}\mapsto\Filt_{\mathcal{T}}(\mathcal{S})$ give one-to-one correspondence between the following two sets.

$(1)$ The set of filtration subcategories $\mathcal{U}$ in $\mathcal{T}$ with $(\mathcal{U},\mathcal{U}^{\perp_{1}})$ being a cotorsion pair.

$(2)$ The set consisting of subsets $\mathcal{S}$ of $\mathcal{X}$ such that $\mathcal{S}_{\mathcal{P}}\subseteq\mathcal{S}$.
\end{theorem}
\begin{proof} Let $\mathcal{U}$ be an arbitrary filtration subcategory in $\mathcal{T}$ such that $(\mathcal{U},\mathcal{U}^{\perp_{1}})$ is a cotorsion pair. Since $\mathcal{P}\subseteq \mathcal{U}=\Filt_{\mathcal{T}}(\mathcal{S_{\mathcal{U}}})$ and then $\Filt_{\mathcal{T}}(\mathcal{S_{\mathcal{P}}})\subseteq\Filt_{\mathcal{T}}(\mathcal{S_{\mathcal{U}}})$, it follows that $\mathcal{S}_{\mathcal{P}}\subseteq\mathcal{S}_{\mathcal{U}}\subseteq\mathcal{X}$ by Lemma \ref{B}; conversely, by Proposition \ref{finite}, we know that $(\Filt_{\mathcal{T}}(\mathcal{S}),\mathcal{S}^{\perp_{1}})$ is a cotorsion pair if $\mathcal{S}_{\mathcal{P}}\subseteq\mathcal{S}\subseteq\mathcal{X}$. On the other hand,  by Lemma \ref{B}, we have that $\Filt_{\mathcal{T}}(\mathcal{S_{\mathcal{U}}})=\mathcal{U}$ if $\mathcal{U}$ is a filtration subcategory; conversely, by Lemma \ref{B} again, we get that $\Filt_{\mathcal{T}}(\mathcal{S}_{\Filt_{\mathcal{T}}(\mathcal{S})})=\Filt_{\mathcal{T}}(\mathcal{S})$, and thus $\mathcal{S}=\mathcal{S}_{\Filt_{\mathcal{T}}(\mathcal{S})}$. This finishes the proof.
\end{proof}

We omit the dual statement of Theorem \ref{H}.

\begin{definition}\label{simple} Let $\mathscr{C}$ be an extriangilated category. A semibrick $\mathcal{X}$ is called a {\em simple-minded system} if $\Filt_{\mathscr{C}}(\mathcal{X})=\mathscr{C}$.
\end{definition}
Note that if $\mathscr{C}$ is a triangulated category, Definition \ref{simple} coincides with \cite[Definition 2.5]{Du}.

\begin{corollary} Let $\mathscr{C}$ be an extriangulated category with enough projectives and injectives. If $\mathcal{X}$ is a simple-minded system of $\mathscr{C}$, the assignments $\mathcal{U}\mapsto\mathcal{S}_{\mathcal{U}}$
and $\mathcal{S}\mapsto\Filt_{\mathscr{C}}(\mathcal{S})$ give one-to-one correspondence between the following two sets.

$(1)$ The set of filtration subcategories $\mathcal{U}$ in $\mathscr{C}$ with $(\mathcal{U},\mathcal{U}^{\perp_{1}})$ being a cotorsion pair.

$(2)$ The set consisting of subsets $\mathcal{S}$ of $\mathcal{X}$ such that $\mathcal{S}_{\mathcal{P}}\subseteq\mathcal{S}$.
\end{corollary}
Let $\mathscr{C}$ be a triangulated category. In this case, if $(\mathcal{U},\mathcal{V})$ is a cotorsion pair in $\mathscr{C}$, then $(\mathcal{U},\mathcal{V}[1])$ is a torsion pair in the sense of \cite[Definition 2.2]{Iy2}. Immediately we have the following
\begin{corollary}\cite[Theorem 3.3]{Du}\label{3jiao} Let $\mathscr{C}$ be a triangulated category, and $\mathcal{X}$ be a simple-minded system of $\mathscr{C}$. Then for any subset $\mathcal{S}\subseteq\mathcal{X}$, $(\Filt_{\mathscr{C}}(\mathcal{S}),\mathcal{S}^{\perp})$ and $(^{\perp}\mathcal{S},\Filt_{\mathscr{C}}(\mathcal{S}))$ are torsion pairs in $\mathscr{C}$.
\end{corollary}
\begin{proof} By Lemma \ref{finite}, we obtain that $(\Filt_{\mathcal{T}}(\mathcal{S}),\mathcal{S}^{\perp_{1}})$ is a cotorsion pair, it follows that $(\Filt_{\mathcal{T}}(\mathcal{S}),\mathcal{S}^{\perp_{1}}[1])$ is a torsion pair. Observe that $\mathcal{S}^{\perp_{1}}[1]=\mathcal{S}^{\perp}$, it means that $(\Filt_{\mathscr{C}}(\mathcal{S}),\mathcal{S}^{\perp})$ is a torsion pair.  It is proved dually that $(^{\perp}\mathcal{S},\Filt_{\mathscr{C}}(\mathcal{S}))$ is a torsion pair.
\end{proof}
\begin{example}\label{W} We consider the hereditary path algebra $A$ of the quiver $1\longrightarrow2\longrightarrow3$. The Auslander--Reiten quiver $\Gamma$ of the bounded derived category $D^{b}(A)$ is given by
\end{example}
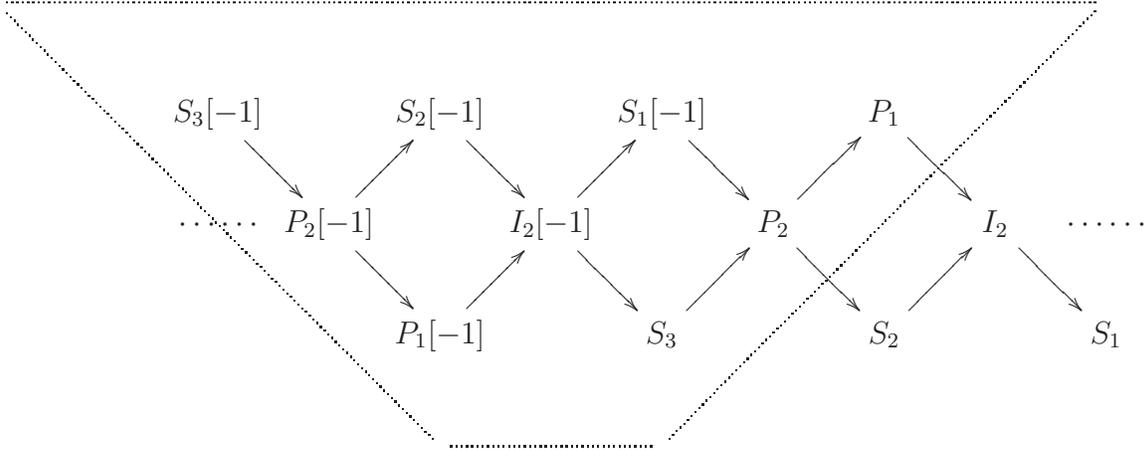
\begin{equation*}
\xymatrix@!=1.5pc{
  \ar@{.}[rrrrrrrrrr]\ar@{.}[ddddrrrr] &&&&&&&&&&\ar@{.}[lllldddd] \\
   && S_3[-1]\ar[dr]  && S_2[-1]\ar[dr] && S_1[-1]\ar[dr] && P_1\ar[dr] && \\
   && \cdots\cdots & P_2[-1]\ar[dr]\ar[ur] && I_2[-1]\ar[ur]\ar[dr] && P_2\ar[ur]\ar[dr] && I_2\ar[dr] & \cdots\cdots\\
   &&&& P_1[-1]\ar[ur] && S_3\ar[ur] && S_2\ar[ur] && S_1 \\
   &&&&\ar@{.}[rr]&&&&&&}
\end{equation*}
Clearly, the set $\mathcal{X}$ consisting of the isoclasses of objects in the top row of $\Gamma$ is a simple-minded system of $D^{b}(A)$. Let $\mathcal{S}$ be the subset of $\mathcal{X}$ which is consisting of the isoclasses of objects in $\{ P_1, S_1[-1],S_2[-1],S_3[-1]\}$. Then $\Filt_{D^{b}(A)}(\mathcal{S})$ is an extriangulated category whose indecomposable objects lie in the trapezoidal area as depicted in $\Gamma$. By Corollary \ref{3jiao}, $(\Filt_{D^{b}(A)}(\mathcal{S}),\mathcal{S}^{\perp})$ and $(^{\perp}\mathcal{S},\Filt_{D^{b}(A)}(\mathcal{S}))$ are torsion pairs in $D^{b}(A)$.

\section*{Acknowledgments}
Supported by the National Natural Science Foundation of China (No.s 11801273, 11771212), Natural Science Foundation of Jiangsu Province of China (No.BK20180722)
and Natural Science Foundation of Jiangsu Higher Education Institutions of China  (No.18KJB110017).



\begin{thebibliography}{99}
\bibitem{As} S. Asai, Semibircks, Int. Math. Res. Not. {\bf 16} (2020), 4993--5054.
\bibitem{Ben} R. Bennett-Tennenhaus, A. Shah, Transport of structure in higher homological algebra, arXiv:2003.02254.
\bibitem{Bu} T. B\"{u}hler, Exact categories, Expo. Math. {\bf 28} (2010), 1--69.

\bibitem{Ch} W. Chang, P. Zhou, B. Zhu, Cluster subalgebras and cotorsion pairs in Frobenius extriangulated categories, Algebr. Represent. Theor. {\bf 22} (2018), 1051--1081.

\bibitem{Du} A. Dugas, Torsion pairs and simple-minded systems in triangulated categories, Appl. Categ. Structures {\bf 23} (2015), 507--526.

\bibitem{En} H. Enomoto, Schur's lemma for exact categories implies abelian, arXiv:2002.09241.
\bibitem{GL} W. Geigle, H. Lenzing, Perpendicular categories with applications to representations and sheaves, J.
Algebra {\bf 144} (1991), 273--343.
\bibitem{Ha} D. Happel, Triangulated categories in the representation theory of finite dimensional algebras, London Mathematical Society Lecture Note Series, vol. 119, Cambridge University Press, Cambridge, 1988.
\bibitem{Iy2} O. Iyama, Y. Yoshino, Mutations in triangulated categories and rigid Cohen-Macaulay modules, Invent. Math.  {\bf 172} (2008), 117--168.
\bibitem{Na} H. Nakaoka, Y. Palu, Extriangulated categories, Hovey twin cotorsion pairs and model structures, Cah. Topol. G\'{e}om. Diff\'{e}r. Cat\'{e}g. {\bf 60}(2) (2019), 117--193.
\bibitem{Ne} A. Neeman, Triangulated categories, Annals of Mathematics Studies, vol. 148, Princeton University Press, Princeton, NJ, 2001.



\bibitem{Ri} C. M. Ringel, Representations of K-species and bimodules, J. Algebra {\bf 41} (1976), 269--302
\bibitem{Sch} A. H. Schofield, Representation of rings over skew fields, London Mathematical Society Lecture Note Series,
vol. 92, Cambridge University Press, Cambridge, 1985.
\bibitem{St} B. Stenstr\"om, Rings of quotients, Die Grundlehren der Mathematischen Wissenschaften, Band 217, An
introduction to methods of ring theory, Springer-Verlag, New York-Heidelberg, 1975.

\bibitem{Th} R. W. Thomason, The classification of triangulated categories, Compositio Math. {\bf 105} (1997), 1--27.
\bibitem{Ve} J. L. Verdier, cat\'{e}gories d\'{e}riv\'{e}s, \'{e}tat0, J. Algebra {\bf 569} (1981), 262--317.

\bibitem{Zh} P. Zhou, Filtered objects in extriangulated categories, Comm. Algebra {\bf 48}(11) (2020), 4580--4595.


\end{thebibliography}
\end{document}